\documentclass[12pt,a4paper,draft]{amsart}

\usepackage{graphicx}
\usepackage{amscd}
\usepackage[mathscr]{eucal}
\usepackage{faktor} 
\usepackage{mathrsfs}
\usepackage{amsmath}
\usepackage{hyperref}
\usepackage[all]{xy}
\usepackage{txfonts}
\usepackage{bigints}
\usepackage{multirow}
\usepackage{mathtools}
\allowdisplaybreaks

\newcommand{\norm}[1]{\left\lVert#1\right\rVert}

\newcommand{\adef}{\begin{defin}}
\newcommand{\zdef}{\end{defin}}
\newtheorem{defin}{Definition}
\newtheorem{theorem}{Theorem}[section]
\newtheorem{lemma}[theorem]{Lemma}
\newtheorem{proposition}[theorem]{Proposition}
\newtheorem{corollary}[theorem]{Corollary}

\numberwithin{equation}{section}

\theoremstyle{definition}

\numberwithin{equation}{section}

\newcommand{\vertiii}[1]{{\left\vert\kern-0.25ex\left\vert\kern-0.25ex\left\vert #1
    \right\vert\kern-0.25ex\right\vert\kern-0.25ex\right\vert}}

\newcommand{\bigslant}[2]{{\raisebox{.2em}{$#1$}\left/\raisebox{-.2em}{$#2$}\right.}}

\def\block(#1,#2)#3{\multicolumn{#2}{c}{\multirow{#1}{*}{$ #3 $}}}

\DeclarePairedDelimiterX{\inp}[2]{\langle}{\rangle}{#1, #2}

\DeclareMathOperator{\Real}{\operatorname{Re}}

\newcommand{\abs}[1]{\lvert#1\rvert}

\theoremstyle{definition}
\newtheorem{definition}[theorem]{Definition}

\theoremstyle{remark}

\numberwithin{equation}{section}

\author{Willian Corrêa}

\thanks{The author was supported by São Paulo Research Foundation (FAPESP), grants 2016/25574-8 and 2021/13401-0.}

\begin{document}

\title[Uniform homeomorphisms induced by interpolation methods]{Uniform homeomorphisms between spheres induced by interpolation methods}

\begin{abstract}
    M. Daher \cite{Daher} showed that if $(X_0, X_1)$ is a regular couple of uniformly convex spaces then the unit spheres of the complex interpolation spaces $X_{\theta}$ and $X_{\eta}$ are uniformly homeomorphic for every $0 < \theta, \eta < 1$. We show that this is a rather general phenomenon of interpolation methods described by the discrete framework of interpolation of \cite{DiscreteFramework}.
\end{abstract}

\maketitle

\section{Introduction}

M. Daher showed in \cite{Daher} that in many natural situations complex interpolation generates uniform homeomorphisms (this result was obtained by Kalton independently \cite[page 216]{GeometricNonlinearBook}). More precisely:

\begin{theorem}[M. Daher]
    Let $(X_0, X_1)$ be a regular compatible couple of uniformly convex Banach spaces. Then for any $\theta, \eta \in (0, 1)$ the spheres of the complex interpolation spaces $X_{\theta}$ and $X_{\eta}$ are uniformly homeomorphic.
\end{theorem}

Daher's Theorem is not exclusive to the complex method. The reiteration theorem between the real and complex methods \cite[Theorem 4.7.2]{BerghLofstrom} shows that the interior of real interpolation scales are preserved by complex interpolation (at least up to equivalence of norms). Therefore, if one starts with a regular compatible couple $(X_0, X_1)$ of uniformly convex spaces and takes $0 < \theta_0 < \theta_1 < 1$ and $1 < p_0 < p_1 < \infty$, a uniform homeomorphism between the spheres of $X_{\theta_0, p_0}$ and $X_{\theta_1, p_1}$ may be found by considering complex interpolation between $X_{\theta_0^{\prime}, p_0^{\prime}}$ and $X_{\theta_1^{\prime}, p_1^{\prime}}$ with $0 < \theta_0^{\prime} < \theta_0 < \theta_1 < \theta_1^{\prime} < 1$ and $0 < p_0^{\prime} < p_0 < p_1 < p_1^{\prime} < \infty$. By \cite[Example 6.6]{DiscreteFramework}, we have similar results for the Rademacher, $\gamma$ and (sometimes) the $\alpha$ methods.

General frameworks of interpolation allow one to see different interpolation methods as particular realizations of the same phenomenon. That is useful to prove results available in one method to others described by the same framework. For example, using the Cwikel-Kalton-Milman-Rochberg method of pseudolattices \cite{CKMR}, Ivtsan showed in \cite{Ivtsan} that Stafney's Lemma holds in many interpolation methods.

Our goal is to show that one does not need to pass trough complex interpolation to obtain Daher's Theorem, but that it is a general feature embedded into the discrete framework of interpolation of Lindemulder and Lorist presented in \cite{DiscreteFramework}. So in principle one does not need a reiteration theorem with the complex method to have uniform homeomorphisms between spheres of interpolation spaces. We use a vector valued version of the James space to give a possible example where reiteration with the complex method fails.

The classical example is the couple $(L_{p_0}, L_{p_1})$ where $1 < p_0 \neq p_1 < \infty$, for which the induced uniform homeomorphisms are the Mazur maps \cite{Mazur1929}. Extending a result of Odell and Schlumprecht \cite{OdellSchlumprecht}, Chaatit proved that the unit sphere of a Banach lattice $X$ with a weak unit which does not contain uniform copies of $\ell_{\infty}^n$ is uniformly homeomorphic to the unit sphere of a Hilbert space \cite{Chaatit}. Daher's Theorem may be used to give a proof of Chaatit's result: an extrapolation theorem of Pisier \cite{PisierApplications} (extended by Kalton in \cite{KaltonKothe}) ensures the existence of a regular compatible couple $(X_0, X_1)$ of uniformly convex spaces such that for some $\theta, \eta \in (0, 1)$ we have $X_{\theta} = X^{(2)}$ (the $2$-convexification of $X$) and $X_{\eta}$ is a Hilbert space. Daher's Theorem then gives a uniform homeomorphism between of the spheres of $X^{(2)}$ and $X_{\eta}$, and all that is left to do is showing that the spheres of $X$ and $X^{(2)}$ are uniformly homeomorphic. See an exposition of these results in \cite[Chapter 9]{GeometricNonlinearBook}. See also \cite{HolderDaher} for a finite-dimensional quantitative version of Daher's Theorem with an application to the Approximation Near Neighbor search.

In light of the facts of the previous paragraph, an extension of Daher's Theorem to other interpolation methods might be useful in the uniform classification of spheres of Banach spaces. In particular, it would be interesting to obtain extrapolation results akin to Pisier's for other methods (or, for that matter, to extend the Kalton Calculus of \cite{KaltonKothe}).

Our proof highlights an often neglected feature of the complex method of interpolation, namely, that associated to a compatible couple $(X_0, X_1)$ one does not have simply an interpolation scale $(X_{\theta})_{0 < \theta < 1}$, but an interpolation family $(X_z)_{0 < \Real(z) < 1}$. Most of the time this is overlooked because $X_z = X_{\Real(z)}$ isometrically, but this need not be true for other methods. Our proof of Daher's Theorem is an evidence that the following comment from \cite{ComplexInterpolationFamilies} is not exclusive to the complex method:

\medskip
``\textit{The current theory of interpolation involves the intermediate spaces between \emph{two} given Banach spaces (the `boundary' spaces). It is our claim that the natural setting for the complex method of interpolation involves a \emph{family} of `boundary' Banach spaces distributed on the boundary, $\partial D$, of a domain in $\mathbb{C}$.}''
\medskip

Daher's Theorem shows that under suitable conditions interpolation changes the spheres of Banach spaces inside the same uniformly homeomorphic class. A related result is that it does so continuously: we show that if the parameters are close to each other then the corresponding interpolation spaces are close with respect to the Kadets metric, generalizing a result of \cite{DistancesBanach}.

It is worth noting that the discrete framework for interpolation of \cite{DiscreteFramework} is at the same time more general than the Cwikel-Kalton-Milman-Rochberg framework of \cite{CKMR} (it does not necessarily define interpolation functors) and more restrictive (all the sequence structures admit differentiation, in the language of \cite{CKMR}). It is clear therefore that our results may be adapted to the framework of \cite{CKMR}.

\section{The discrete framework for interpolation}

For background on interpolation spaces we refer the reader to \cite{BerghLofstrom}. The authors of \cite{DiscreteFramework} provided a general framework that encompasses many interpolation methods. We describe it now with some adaptations: first, instead of restricting our interpolation parameter $\theta$ to $(0, 1)$ we allow it to be any complex number with real part in $(0, 1)$. Second, instead of taking a maximum for the norm of intersection spaces we take an $\ell_2$-norm. It is clear that this last change will give the same interpolation spaces up to equivalence of norms, which in turn will not impact the existence of uniform homeomorphisms with respect to the original norms.

Let $X$ be a complex Banach space, and let $\ell^0(\mathbb{Z}; X)$ be the space of $X$-valued sequences. A \textit{sequence structure} on $X$ is Banach space $\mathfrak{S}$ contained in $\ell^0(\mathbb{Z}; X)$ which is translation invariant and for which we have norm $1$ inclusions
\[
\ell^1(\mathbb{Z}; X) \subset \mathfrak{S} \subset \ell^{\infty}(\mathbb{Z}; X)
\]

\noindent The couple $\mathcal{X} = [X, \mathfrak{S}]$ is called a \textit{sequentially structured Banach space}. If $a \in \mathbb{C}^*$ we let $\mathfrak{S}(a)$ be the space of all $\vec{x} = (x_k) \in \ell^0(\mathbb{Z}, X)$ such that
\[
\|\vec{x}\|_{\mathfrak{S}(a)} = \|(a^k x_k)\|_{\mathfrak{S}} < \infty
\]

Suppose we have a compatible couple $(X_0, X_1)$ of Banach spaces such that each $X_j$ is a sequentially structure Banach space, i.e., we have sequence structures $\mathcal{X}_j = [X_j, \mathfrak{S}_j]$. The couple $(\mathcal{X}_0, \mathcal{X}_1)$ is called a \textit{compatible couple of sequentially structured Banach spaces}.

If $(X_0, X_1)$ is a compatible couple we denote by $X_0 \cap_2 X_1$ the space $X_0 \cap X_1$ with the equivalent norm $\|x\|_{X_0 \cap_2 X_1} = (\|x\|_{X_0}^2 + \|x\|_{X_1}^2)^{\frac{1}{2}}$.

Let $\mathbb{S} = \{z \in \mathbb{C} : 0 < \Real(z) < 1\}$ and let $z \in \mathbb{S}$. For $x \in X_0 + X_1$ we let
\[
\|x\|_{z,2} = \inf \|\vec{x}\|_{\mathfrak{S}_0(e^{-z}) \cap_2 \mathfrak{S}_1(e^{1-z})} = \inf \; (\|\vec{x}\|_{\mathfrak{S}_0(e^{-z})}^2 + \|\vec{x}\|_{\mathfrak{S}_1(e^{1-z})}^2)^{\frac{1}{2}}
\]
where the infimum is over all sequences $\vec{x} = (x_k) \in \mathfrak{S}_0(e^{-z}) \cap \mathfrak{S}_1(e^{1-z})$ such that $x = \sum\limits_{k \in \mathbb{Z}} x_k$ in $X_0 + X_1$. Define
\[
(\mathcal{X}_0, \mathcal{X}_1)_z^{(2)} = \mathcal{X}_z^{(2)} = \{x \in X_0 + X_1 : \|x\|_{z,2} < \infty\}
\]
This definition recovers the spaces $(\mathcal{X}_0, \mathcal{X}_1)_{\theta}$ of \cite{DiscreteFramework} for $0 < \theta < 1$ (with an equivalent norm). One may check that the spaces $(\mathcal{X}_0, \mathcal{X}_1)_z^{(2)}$ satisfy an interpolation estimate like the one of \cite[Theorem 5.2]{DiscreteFramework}, substituting $e^{\theta}$ by $e^{\Real(z)}$.

\section{Optimal representations}

One important step in the proof of Daher's Theorem for the complex method is the following result on optimal representations \cite[Proposition 3]{Daher}:

\begin{theorem}\label{thm:optimal_representations_Daher}
Let $(X_0, X_1)$ be a regular compatible couple (i.e., $X_0 \cap X_1$ is dense in $X_0$ and $X_1$) of reflexive Banach spaces. Let $\theta \in (0, 1)$.
\begin{enumerate}
    \item If $x \in S_{X_{\theta}}$ then there is $g$ in the Calderón space $\mathcal{H}^{\infty}(X_0, X_1)$ such that $g(\theta) = x$ and $\|g(j + it)\|_{X_j} = 1$ for almost every $t$, $j = 0, 1$. In particular, $\|g\| = \|x\|_{\theta}$.
    \item If $X_0$ is strictly convex then $g$ is unique with the property that $g(\theta) = x$ and $\|g\| = \|x\|_{\theta}$.
\end{enumerate}
\end{theorem}

Let $(\mathcal{X}_0, \mathcal{X}_1)$ be a compatible couple of sequentially structure Banach spaces, $z \in \mathbb{S}$, and consider the map
\[
\Sigma : \mathfrak{S}_0(e^{-z}) \cap_2 \mathfrak{S}_1(e^{1 - z}) \rightarrow X_0 + X_1
\]
given by $\Sigma(\vec{x}) = \sum_{k \in \mathbb{Z}} x_k$. Following \cite[Remark 3.2]{DiscreteFramework}, it is possible to prove that $\Sigma$ is well-defined and bounded. It follows directly from the definitions that we have an induced isometry
\[
\tilde{\Sigma} : \bigslant{\mathfrak{S}_0(e^{-z}) \cap_2 \mathfrak{S}_1(e^{1 - z})}{\ker \; \Sigma} \rightarrow (\mathcal{X}_0, \mathcal{X}_1)_z^{(2)}
\]

\begin{definition}
A sequentially structure Banach space $\mathcal{X} = [X, \mathfrak{S}]$ will be called \emph{reflexive/strictly convex/uniformly convex} if so is $\mathfrak{S}$.
\end{definition}

Notice that if $\mathcal{X} = [X, \mathfrak{S}]$ is a sequentially structured Banach space then $X$ is a subspace of $\mathfrak{S}$, and therefore if $\mathcal{X}$ is reflexive/strictly convex/uniformly convex then so is $X$. We at once get the following result:

\begin{theorem}\label{thm:first_optimal}
Let $(\mathcal{X}_0, \mathcal{X}_1)$ be a compatible couple of sequentially structured Banach spaces and let $z \in \mathbb{S}$.
\begin{enumerate}
    \item If $\mathcal{X}_0$ and $\mathcal{X}_1$ are reflexive then given $x \in S_{\mathcal{X}_z^{(2)}}$ there is $\vec{x} \in \mathfrak{S}_0(e^{-z}) \cap_2 \mathfrak{S}_1(e^{1 - z})$ of norm $1$ such that $\Sigma(\vec{x}) = x$.
    \item If $\mathcal{X}_0$ and $\mathcal{X}_1$ are strictly convex then the previous element $\vec{x}$ is unique.
\end{enumerate}
\end{theorem}

To continue Daher uses properties of the function $g$ of Theorem \ref{thm:optimal_representations_Daher}. We shall therefore use the complex description of the discrete framework of interpolation. If $X$ is a Banach space let $\mathcal{H}(\mathbb{S}, X)$ be the space of analytic $X$-valued functions on $\mathbb{S}$. We will say that $f \in \mathcal{H}(\mathbb{S}, X)$ is $2\pi$-periodic if $f(z + 2\pi i) = f(z)$ for every $z \in \mathbb{S}$ and let $f_z(t) = f(z + it)$ be defined for $t \in \mathbb{R}$..

Let us consider the space $\mathcal{H}_{\pi}(\mathbb{S}, X)$ of $2\pi$-periodic functions in $\mathcal{H}(\mathbb{S}, X)$. For those functions it makes sense to take Fourier coefficients:
\[
\hat{f_z}(k) = \frac{1}{2\pi} \int_0^{2\pi} f(z + it) e^{-ikt} dt
\]
for $k \in \mathbb{Z}$. According to \cite[Lemma 4.1]{DiscreteFramework}, the sequence $(e^{-ks} \hat{f_s}(k))_{k \in \mathbb{Z}}$ is independent of $s \in (0, 1)$. Similarly, we have:

\begin{lemma}
The sequence $(e^{-kz} \hat{f}_z (k))_{k \in \mathbb{Z}}$ is independent of $z \in \mathbb{S}$.
\end{lemma}

This means that the following definition is independent of $z_0 \in \mathbb{S}$. Let $(\mathcal{X}_0, \mathcal{X}_1)$ be a compatible couple of sequentially structured Banach spaces and define $\mathcal{H}^2_{\pi}(\mathbb{S}, \mathcal{X}_0, \mathcal{X}_1)$ as the space of all $f \in \mathcal{H}_{\pi}(\mathbb{S}, \mathcal{X}_0 + \mathcal{X}_1)$ such that
\[
\|f\|_{\mathcal{H}^2_{\pi}} \coloneqq (\|\hat{f}_{z_0}\|_{\mathfrak{S}_0(e^{-z_0})}^2 + \|\hat{f}_{z_0}\|_{\mathfrak{S}_1(e^{1-z_0})}^2)^{\frac{1}{2}} < \infty
\]

As in \cite[Lemma 4.2]{DiscreteFramework}, the map $f \mapsto \hat{f}_{z_0}$ is an isometric isomorphism from $\mathcal{H}^2_{\pi}(\mathbb{S}, \mathcal{X}_0, \mathcal{X}_1)$ onto $\mathfrak{S}_0(e^{-z_0}) \cap_2 \mathfrak{S}_1(e^{1 - z_0})$. Its inverse is given by
\[
\vec{x} \mapsto f(z) = \sum\limits_{k \in \mathbb{Z}} e^{k(z - z_0)} x_k
\]
It follows at once that
\begin{equation}\label{eq:isometry_quotient_functions}
    \|x\|_{z_0, 2} = \inf \{\|f\|_{\mathcal{H}^2_{\pi}} : f \in \mathcal{H}^2_{\pi}(\mathbb{S}, \mathcal{X}_0, \mathcal{X}_1), f(z_0) = x\}
\end{equation}

Let $\delta_{z_0} : \mathcal{H}^2_{\pi}(\mathbb{S}, \mathcal{X}_0, \mathcal{X}_1) \rightarrow X_0 + X_1$ be given by $\delta_{z_0}(f) = f(z_0)$. 

\begin{lemma}\label{lem:isometry_quotient_functions}
$\delta_{z_0}$ is bounded and $\mathcal{X}_{z_0}^{(2)} = \bigslant{\mathcal{H}^2_{\pi}(\mathbb{S}, \mathcal{X}_0, \mathcal{X}_1)}{\ker \; \delta_{z_0}}$ isometrically.
\end{lemma}
\begin{proof}
    The first part follows from the commutative diagram
    \[
        \xymatrix{\mathcal{H}_{\pi}^2(\mathbb{S}, \mathcal{X}_0, \mathcal{X}_1) \ar[r]\ar[d]^{\delta_{z_0}} & \mathfrak{S}_0(e^{-z_0}) \cap_2 \mathbb{S}_1(e^{1 - z_0}) \ar[d]^{\Sigma} \\
        X_0 + X_1 \ar@{=}[r] & X_0 + X_1}
    \]
    where the horizontal arrow is the isometry described above. The second part of the result is simply \ref{eq:isometry_quotient_functions}.
\end{proof}

We get at once:
\begin{theorem}\label{thm:Daher1}
Let $(\mathcal{X}_0, \mathcal{X}_1)$ be a compatible couple of sequentially structured Banach spaces and $z_0 \in \mathbb{S}$.
\begin{enumerate}
    \item If $\mathcal{X}_0$ and $\mathcal{X}_1$ are reflexive, then given $x \in S_{\mathcal{X}_{z_0}^{(2)}}$ there is $f \in \mathcal{H}^2_{\pi}(\mathbb{S}, \mathcal{X}_0, \mathcal{X}_1)$ of norm $1$ such that $f(z_0) = x$.
    \item If $\mathcal{X}_0$ and $\mathcal{X}_1$ are strictly convex then the previous element $f$ is unique, and we denote it by $\Gamma_{z_0}(x)$ and call it \emph{the optimal function associated to $x$}.
\end{enumerate}
\end{theorem}

Our goal is to show that, for any $z \in \mathbb{S}$, $x \mapsto \Gamma_{z_0}(x)(z)$ is a uniform homeomorphism between the unit spheres of $\mathcal{X}_{z_0}^{(2)}$ and $\mathcal{X}_{z}^{(2)}$. For that we need to show that $\|\Gamma_{z_0}(x)(z)\|_{\mathcal{X}_{z}^{(2)}} = 1$, what will be done via duality.

\section{Duality}

In \cite{DiscreteFramework} the dual of the interpolation spaces is described up to equivalence of norms. We will need an isometric description. For that we will have consider more properties of sequence structures. A sequentially structured Banach space $\mathcal{X} = [X, \mathfrak{S}]$ is called \emph{reflection invariant} if for every $(x_k)_{k \in \mathbb{Z}} \in \mathfrak{S}$ we have $\|(x_k)_{k \in \mathbb{Z}}\|_{\mathfrak{S}} = \|(x_{-k})_{k \in \mathbb{Z}}\|_{\mathfrak{S}}$. If $\lim_{n \rightarrow \infty} C_n \vec{x} = \vec{x}$ for every $\vec{x} \in \mathfrak{S}$, where $C_n$ is the Cesàro operator
\[
C_n \vec{x} = \frac{1}{n+1} \sum\limits_{m=0}^n (\cdots, 0, x_{-m}, \cdots, x_m, 0, \cdots)
\]
and $\sup_n \|C_n \vec{x}\|_{\mathfrak{S}} \leq \|\vec{x}\|_{\mathfrak{S}}$ then $\mathcal{X}$ is called a \emph{$c_0$-sequentially structured Banach space}. If $\mathcal{X}$ is a $c_0$-sequentially structured Banach space and $a \in \mathbb{C}$ then $\mathcal{X}(a)^* = \mathcal{X}^*(a^{-1})$ isometrically (see the commentary after \cite[Lemma 3.13]{DiscreteFramework}).

In this section we will fix a regular couple $(\mathcal{X}_0, \mathcal{X}_1)$ of $c_0$-sequentially structured Banach spaces which are reflection invariant. Recall that we have
\[
\mathcal{X}_z^{(2)} = \bigslant{\mathfrak{S}_0(e^{-z}) \cap_2 \mathfrak{S}_1(e^{1 - z})}{\ker \; \Sigma}
\]
Notice that $\mathfrak{S}_0(e^{-z}) \cap_2 \mathfrak{S}_1(e^{1 - z})$ is a closed subspace of $\mathfrak{S}_0(e^{-z}) \oplus_2 \mathfrak{S}_1(e^{1 - z})$. Under this identification,
\[
\ker \; \Sigma = \{(\vec{x}, \vec{x}) \in \mathfrak{S}_0(e^{-z}) \oplus_2 \mathfrak{S}_1(e^{1-z}) : \Sigma \vec{x} = 0\}
\]
and from the continuity of $\Sigma$ on $\mathfrak{S}_0(e^{-z}) \cap_2 \mathfrak{S}_1(e^{1 - z})$ we have that $\ker \Sigma$ is closed in $\mathfrak{S}_0(e^{-z}) \oplus_2 \mathfrak{S}_1(e^{1 - z})$. It follows that
\[
(\mathcal{X}_z^{(2)})^* = \bigslant{(\ker \; \Sigma)^{\perp}}{(\mathfrak{S}_0(e^{-z}) \cap_2 \mathfrak{S}_1(e^{1 - z}))^{\perp}}
\]
isometrically, where the annihilators are taken inside $(\mathfrak{S}_0(e^{-z}) \oplus_2 \mathfrak{S}_1(e^{1 - z}))^*$. We have:
\[
(\mathfrak{S}_0(e^{-z}) \oplus_2 \mathfrak{S}_1(e^{1 - z}))^* = \mathfrak{S}_0^*(e^{z}) \oplus_2 \mathfrak{S}_1^*(e^{z - 1})
\]
and by reflection invariance
\[
\mathfrak{S}_0^*(e^{z}) \oplus_2 \mathfrak{S}_1^*(e^{z - 1}) = \mathfrak{S}_0^*(e^{-z}) \oplus_2 \mathfrak{S}_1^*(e^{1 - z})
\]
The sum operator $\Sigma$ does not care about reflection, therefore $(\ker \; \Sigma)^{\perp}$ is still the same.

\begin{lemma}
$(\ker \; \Sigma)^{\perp} = \{(\vec{x^*}, \vec{y^*}) : \exists \; x^* \in X_0^* + X_1^* : x_k^* + y_k^* = x^* \; \forall k \}$.
\end{lemma}
\begin{proof}
Let $(\vec{x^*}, \vec{y^*}) \in (\ker \; \Sigma)^{\perp}$, and let $x \in X_0 \cap X_1$. Let $j \in \mathbb{Z}^*$. Consider the sequence $\vec{z} = (z_k)_{k \in \mathbb{Z}}$ such that $z_0 = x$, $z_{-j} = -x$, and $z_k = 0$ otherwise. Then $\vec{z} \in \ker \; \Sigma$, and $0 = (\vec{x^*}, \vec{y^*})(\vec{z}) = (x^*_0 + y^*_0 - x^*_j - y^*_j)(x)$. Since $X_0 \cap X_1$ is dense in $X_0$ and $X_1$, it follows that $x^*_0 + y^*_0 = x^*_j + y^*_j$ for every $j$. The other inclusion is clear.
\end{proof}

\begin{lemma}
$(\mathfrak{S}_0(e^{-z}) \cap_2 \mathfrak{S}_1(e^{1 - z}))^{\perp} = \{(\vec{x^*}, -\vec{x^*}) : \vec{x^*} \in \mathfrak{S}_0(e^{-z})^* \cap \mathfrak{S}_1(e^{1 - z})^*\}$.
\end{lemma}
\begin{proof}
Let $(\vec{x^*}, \vec{y^*}) \in (\mathfrak{S}_0(e^{-z}) \cap_2 \mathfrak{S}_1(e^{1 - z}))^{\perp}$. Given any $x \in X_0 \cap X_1$ and $j \in \mathbb{Z}$, let $\vec{z}$ be the sequence such that $z_{-j} = x$, and $z_k = 0$ otherwise. Then $0 = (\vec{x^*}, \vec{y^*})(\vec{z}) = (x^*_j + y^*_j)(x)$. It follows that $\vec{x^*} = -\vec{y^*}$. The other inclusion is clear. 
\end{proof}

The previous results motivate the following definition, which already appears for $z \in (0, 1)$ with an equivalent norm in \cite[Section 3.3]{DiscreteFramework}:

\begin{definition}
For $z \in \mathbb{S}$ let
\[
\mathcal{X}_z^{(2), m} = \{x \in X_0 + X_1 : \|(\cdots, x, x, x, \cdots)\|_{\mathfrak{S}_0(e^{-z}) +_2 \mathfrak{S}_1(e^{1 - z})} < \infty\}
\]
\end{definition}

\begin{theorem}
Let $(\mathcal{X}_0, \mathcal{X}_1)$ be a regular couple of reflection invariant $c_0$-sequentially structured Banach spaces. Then $(\mathcal{X}_z^{(2)})^* = (\mathcal{X}^*)_z^{(2), m}$ isometrically.
\end{theorem}
\begin{proof}
We have seem that
\begin{eqnarray*}
(\mathcal{X}_z^{(2)})^* & = & \bigslant{(\ker \; \Sigma)^{\perp}}{(\mathfrak{S}_0(e^{-z}) \cap_2 \mathfrak{S}_1(e^{1 - z}))^{\perp}} \\
& = & \bigslant{\{(\vec{x^*}, \vec{y^*}) : \exists \; x^* \in X_0^* + X_1^* : x_k^* + y_k^* = x^* \; \forall k \}}{\{(\vec{x^*}, -\vec{x^*}) : \vec{x^*} \in \mathfrak{S}_0(e^{-z})^* \cap \mathfrak{S}_1(e^{1 - z})^*\}}
\end{eqnarray*}
inside $\mathfrak{S}_0^*(e^{-z}) \oplus_2 \mathfrak{S}_1^*(e^{1 - z})$ isometrically. Notice that $(\vec{x^*}, \vec{y^*})$ and $(\vec{r^*}, \vec{t^*})$ are equivalent if and only if there is $x^* \in X_0^* + Y_0^*$ such that $x^* = x_k^* + y_k^* = r_k^* + t_k^*$ for every $k$. Therefore
\[
\|x^*\|_{(\mathcal{X}_z^{(2)})^*} = \inf (\|\vec{x^*}\|_{\mathfrak{S}_0^*(e^{-z})}^2 + \|\vec{y^*}\|_{\mathfrak{S}_1^*(e^{1 - z})}^2)^{\frac{1}{2}}
\]
where the infimum is over all $\vec{x^*}, \vec{y^*}$ in the indicated spaces such that $x^*_k + y^*_k = x^*$ for every $k$. That is the norm of $x^*$ in $(\mathcal{X}^*)_z^{(2), m}$.
\end{proof}

We at once get existence of an optimal representation for elements of $(\mathcal{X}_z^{(2)})^*$:

\begin{lemma}
Let $(\mathcal{X}_0, \mathcal{X}_1)$ be a regular couple of reflexive reflection invariant $c_0$-sequentially structured Banach spaces and let $z \in \mathbb{S}$. Then every element $x$ of $(\mathcal{X}_z^{(2)})^*$ admits a representation $(\vec{x^*}, \vec{y^*})$ such that $x^* = x_k^* + y_k^*$ for every $k$ and
\[
\|x^*\|_{(\mathcal{X}^*)_z^{(2), m}} = (\|\vec{x^*}\|_{\mathfrak{S}_0^*(e^{-z})}^2 + \|\vec{y^*}\|_{\mathfrak{S}^*_1(e^{1 - z})}^2)^{\frac{1}{2}}
\]
\end{lemma}

\section{Daher's Theorem}

To prove Daher's Theorem, we will use optimal representations to associate to an element of the dual of $\mathcal{X}_z^{(2)}$ an analytic function. For the next two results we consider a regular couple $(\mathcal{X}_0, \mathcal{X}_1)$ of reflexive reflection invariant $c_0$-sequentially structured Banach spaces.

\begin{lemma}\label{lem:def_A}
Let $(\mathcal{X}_0, \mathcal{X}_1)$ be a couple of reflection invariant $c_0$-sequentially structured Banach spaces. Let $x^* \in (\mathcal{X}_z^{(2)})^*$ and take $\vec{x^*}, \vec{y^*}$ such that $x^* = x_k^* + y_k^*$ for every $k$. Define $Ax^*$ by $(Ax^*)_k = x_k^* - x_{k-1}^* = -(y_k^* - y_{k-1}^*) \in X_0^* \cap X_1^*$. Then $\Sigma Ax^* = x^*$ and $Ax^* \in \mathfrak{S}_0^*(e^{-z}) \cap \mathfrak{S}_1^*(e^{1 - z})$.
\end{lemma}
\begin{proof}
It is similar to the case $z \in (0, 1)$, which is in the proof of \cite[Theorem 3.12]{DiscreteFramework}.
\end{proof}

\begin{lemma}\label{lem:def_g}
Let $(\mathcal{X}_0, \mathcal{X}_1)$ be a regular couple of reflexive reflection invariant $c_0$-sequentially structured Banach spaces. Let $x^* \in (\mathcal{X}_{z_0}^{(2)})^*$ and take $\vec{x^*}, \vec{y^*}$ such that $x^* = x_k^* + y_k^*$ for every $k$ and
\[
\|x^*\|_{(\mathcal{X}^*)_z^{(2), m}} = (\|\vec{x^*}\|_{\mathfrak{S}_0^*(e^{-z})}^2 + \|\vec{y^*}\|_{\mathfrak{S}^*_1(e^{1 - z})}^2)^{\frac{1}{2}}
\]
Define $Ax^*$ as in Lemma \ref{lem:def_A} and
\[
g(z) = \sum\limits_{k \in \mathbb{Z}} e^{k(z - z_0)} (Ax^*)_k
\]
Then for every $z \in \mathbb{S}$ we have
\[
\|g(z)\|_{(\mathcal{X}_z^{(2)})^*} \leq \|x\|_{(\mathcal{X}_{z_0}^{(2)})^*}
\]
\end{lemma}
\begin{proof}
We have
\[
g(z) = \sum\limits_{k \in \mathbb{Z}} e^{k(z - z_0)} (Ax^*)_k = \sum\limits_{k \in \mathbb{Z}} e^{k(z - z_0)} (x_k^* - x_{k-1}^*) = - \sum\limits_{k \in \mathbb{Z}} e^{k(z - z_0)} (y_k^* - y_{k-1}^*)
\]

Let $a_n = \sum\limits_{k = -\infty}^n e^{k(z - z_0)} (x_k^* - x_{k-1}^*)$ and $b_n = \sum\limits_{k=n+1}^{\infty} e^{k(z - z_0)} (x_k^* - x_{k-1}^*)$. Then $a_n + b_n = g(z)$ for every $n$ and
\begin{eqnarray*}
\|\vec{a}\|_{\mathfrak{S}_0^*(e^{-z})} & = & \|(\sum\limits_{k=-\infty}^n e^{-kz_0} (x_k^* - x_{k-1}^*))_n\|_{\mathfrak{S}_0^*} \\
& = & \|(\sum\limits_{k=-\infty}^n (x_k^* - x_{k-1}^*))_n\|_{\mathfrak{S}_0^*(e^{-z_0})} \\
& = & \|(x_{n}^*)_n\|_{\mathfrak{S}_0^*(e^{-z_0})}
\end{eqnarray*}

Similarly, $\|\vec{b}\|_{\mathfrak{S}_1^*(e^{1-z})} = \|(y_n^*)_n\|_{\mathfrak{S}_1^*(e^{1-z_0})}$. Therefore $\|g(z)\|_{(\mathcal{X}_z^{(2)})^*} \leq \|x\|_{(\mathcal{X}_{z_0}^{(2)})^*}$.
\end{proof}

\begin{definition}
A regular couple of uniformly convex reflection invariant $c_0$-sequentially structured Banach spaces is called a \emph{Daher couple}.
\end{definition}

\begin{theorem}\label{thm:existence_of_map_spheres}
Let $(\mathcal{X}_0, \mathcal{X}_1)$ be a Daher couple and $z_0 \in \mathbb{S}$. Let $\Gamma_{z_0} : \mathcal{X}_{z_0}^{(2)} \rightarrow \mathcal{H}_{\pi}^2(\mathbb{S}, \mathcal{X}_0, \mathcal{X}_1)$ be the application that sends $x$ to its optimal function (see Theorem \ref{thm:Daher1}). Then for every $z \in \mathbb{S}$ we have $\|\Gamma_{z_0}(x)(z)\|_{z, 2} = \|x\|_{z_0, 2}$.
\end{theorem}
\begin{proof}
Let $x \in \mathcal{X}_{z_0}^{(2)}$.
Take $x^* \in (\mathcal{X}_{z_0}^{(2)})^* = (\mathcal{X^*})_{z_0, 2}^m$ such that $\|x^*\|_{(\mathcal{X^*})_{z_0, 2}^m} = \|x\|_{\mathcal{X}_{z_0}^{(2)}} = x^*(x) = 1$ and an optimal representation $\vec{x^*}, \vec{y^*}$ of $x^*$, that is, $x^* = x_k^* + y_k^*$ for every $k$ and
\[
1 = \|x^*\|_{(\mathcal{X^*})_{z_0, 2}^m} = (\|\vec{x^*}\|_{\mathfrak{S}_0^*(e^{-z_0})}^2 + \|\vec{y^*}\|_{\mathfrak{S}_1^*(e^{1-z_0})}^2)^{\frac{1}{2}}
\]

Let $Ax^*$ be as in Lemma \ref{lem:def_A} and $g$ be as in Lemma \ref{lem:def_g}. Consider the function $F : \mathbb{S} \rightarrow \mathbb{C}$ given by
\[
F(z) = \langle g(z), \Gamma_{z_0}(x)(z) \rangle
\]

Since the series that defines $F(z)$ converges uniformly on compact subsets of $\mathbb{S}$, $F$ is analytic (see the proof of \cite[Lemma 4.2]{DiscreteFramework}). Also, $F(z_0) = 1$ and for every $z \in \mathbb{S}$
\begin{eqnarray*}
\abs{F(z)} & = & \abs{\langle g(z), \Gamma_{z_0}(x)(z) \rangle} \\
& \leq & \|g(z)\|_{(\mathcal{X}_{z_0}^{(2)})^*} \|\Gamma_{z_0}(x)(z)\|_{\mathcal{X}_{z_0}^{(2)}} \\
& \leq & 1
\end{eqnarray*}
because of Lemma \ref{lem:def_g}. By the Maximum Modulus Principle, $F \equiv 1$. The result follows.
\end{proof}

\textbf{Observation:} For the previous result we only needed reflexivity and strict convexity, not uniform convexity.
\medskip

By Theorem \ref{thm:existence_of_map_spheres}, given different $z, w \in \mathbb{S}$ the interpolation process induces a map from the sphere of $\mathcal{X}_z^{(2)}$ into the sphere of $\mathcal{X}_w^{(2)}$. We will prove that this map is a uniform homeomorphism.

\begin{theorem}\label{thm:optimal_function_is_uniform}
Let $(\mathcal{X}_0, \mathcal{X}_1)$ be a Daher couple and let $z \in \mathbb{S}$. Let $\Gamma_{z} : S_{\mathcal{X}_z^{(2)}} \rightarrow \mathcal{H}_{\pi}^2(\mathbb{S}, \mathcal{X}_0, \mathcal{X}_1)$ be the application that sends $x$ to its optimal function. Then $\Gamma_{z}$ is uniformly continuous.
\end{theorem}
\begin{proof}
Let $\delta$ be the modulus of convexity of $\mathcal{H}_{\pi}^2(\mathbb{S}, \mathcal{X}_0, \mathcal{X}_1)$. For $x, y \in S_{\mathcal{X}_z^{(2)}}$ let $f_x = \Gamma_{z}(x)$ and $f_y = \Gamma_{z}(y)$. It follows that
\[
\norm{\frac{f_x + f_y}{2}}_{\mathcal{H}_{\pi}^2(\mathbb{S}, \mathcal{X}_0, \mathcal{X}_1)} \leq 1 - \delta(\|f_x - f_y\|_{\mathcal{H}_{\pi}^2(\mathbb{S}, \mathcal{X}_0, \mathcal{X}_1)})
\]

So
\[
\norm{\frac{x + y}{2}}_{\mathcal{X}_{z_0}^{(2)}} \leq 1 - \delta(\|f_x - f_y\|_{\mathcal{H}_{\pi}^2(\mathbb{S}, \mathcal{X}_0, \mathcal{X}_1)})
\]

Therefore
\begin{eqnarray*}
\norm{\frac{x - y}{2}}_{\mathcal{X}_{z_0}^{(2)}} & \geq & 1 - \norm{\frac{x + y}{2}}_{\mathcal{X}_{z_0}^{(2)}} \\
& \geq & 1 - \Big(1 - \delta(\|f_x - f_y\|_{\mathcal{H}_{\pi}^2(\mathbb{S}, \mathcal{X}_0, \mathcal{X}_1)})\Big) \\
& = & \delta(\|f_x - f_y\|_{\mathcal{H}_{\pi}^2(\mathbb{S}, \mathcal{X}_0, \mathcal{X}_1)})
\end{eqnarray*}

If we let $\beta(t) = \sup\{u \geq 0 : \delta(u) \leq t\}$ then
\[
\|f_x - f_y\|_{\mathcal{H}_{\pi}^2(\mathbb{S}, \mathcal{X}_0, \mathcal{X}_1)} \leq \beta\Big(\norm{x-y}_{\mathcal{X}_{z_0}^{(2)}}/2\Big)
\]

Since $\lim_{t \rightarrow 0} \beta(t) = 0$, it follows that $\Gamma_{z}$ is uniformly continuous.
\end{proof}

\begin{theorem}[Daher's Theorem]
Let $(\mathcal{X}_0, \mathcal{X}_1)$ be Daher a couple and let $z, w \in \mathbb{S}$. Let $\Gamma_{z} : S_{X_{z, 2}} \rightarrow \mathcal{H}_{\pi}^2(\mathbb{S}, \mathcal{X}_0, \mathcal{X}_1)$ be the application that sends $x$ to its optimal function. Then $U_{z, w} : S_{\mathcal{X}_z^{(2)}} \rightarrow S_{\mathcal{X}_w^{(2)}}$ given by $U_{z,w}(x) = \Gamma_{z}(x)(w)$ is a uniform homeomorphism.
\end{theorem}
\begin{proof}
The map $U_{z, w}$ is surjective because its inverse is $U_{w, z}$. Since
\[
\|U_{z,w}(x) - U_{z,w}(y)\|_{\mathcal{X}_w^{(2)}} = \|\Gamma_{z}(x)(w) - \Gamma_{z}(y)(w)\|_{\mathcal{X}_w^{(2)}} \leq \|\Gamma_{z}(x) - \Gamma_{z}(y)\|_{\mathcal{H}_{\pi}^2(\mathbb{S}, \mathcal{X}_0, \mathcal{X}_1)}
\]
it follows from Theorem \ref{thm:optimal_function_is_uniform} that $U_{z,w}$ is uniformly continuous. Similarly, $U_{w,z}$ is uniformly continuous.
\end{proof}

Notice that we are using the same space of functions to define all interpolation spaces. This is not what happens in Daher's formulation of the complex method. Indeed, for every $z \in \mathbb{S}$ there is a function space $\mathcal{F}_{z}$ used to define the interpolation space. This is responsible for the following: in principle one may only bound the modulus of continuity of the maps $U_{z, w}$ for $0 < a < \Real(z), \Real(w) < b < 1$ (see \cite[Proposition 9.13]{GeometricNonlinearBook}). Since in the discrete formulation all interpolation spaces are defined through the same function space, we get:
\begin{proposition}
Let $(\mathcal{X}_0, \mathcal{X}_1)$ be a Daher couple. Then there is a map $\gamma$ satisfying $\lim_{\epsilon \rightarrow 0} \gamma(\epsilon) = 0$ such that the modulus of continuity of $U_{z, w}$ is bounded by $\gamma$ for every $z, w \in \mathbb{S}$.
\end{proposition}
Of course, the previous result is isometric in nature, and therefore depends on the particular representation of the discrete framework of interpolation.

We will call an interpolation method a \emph{Daher method} if for every regular compatible couple $(X_0, X_1)$ of uniformly convex spaces the associated sequentially structured couple $(\mathcal{X}_0, \mathcal{X}_1)$ is a Daher couple. We now give some examples.

Through the rest of this section we let $(X_0, X_1)$ be a regular compatible couple of uniformly convex Banach spaces.

\subsection{Real method} Let $p_0, p_1 \in (1, \infty)$ and consider the sequence structures $\mathfrak{S}_j = \ell^{p_j}(\mathbb{Z}, X_j)$. It is clear that this gives us a Daher method. We have that $(\mathcal{X}_0, \mathcal{X}_1)_z$ is the interpolation space $(X_0, X_1)_{\Real(z), p_0, p_1}$ given by the Lions-Peetre mean method of \cite{Lions_Peetre_64} with equivalence of norms. In turn, that is equal to the real interpolation space $(X_0, X_1)_{\Real(z), p}$, where $\frac{1}{p} = \frac{1-\Real(z)}{p_0} + \frac{\Real(z)}{p_1}$.

\subsection{Complex method} Since we are dealing with reflexive spaces, the lower and upper complex methods agree, so we only need to describe the lower method. Let $p_0, p_1 \in (1, \infty)$ and consider the sequence structures $\mathfrak{S}_j = \hat{L}^{p_j}(\mathbb{T}, X_j)$ of Fourier coefficients $\hat{f} = (\hat{f}(k))_{k \in \mathbb{Z}}$ of functions in $L^{p_j}(\mathbb{T}, X_j)$, with $\|\hat{f}\|_{\hat{L}^{p_j}(\mathbb{T}, X_j)} = (2\pi)^{\frac{1}{p}} \|f\|_{L^{p_j}(\mathbb{T}, X_j)}$. Again, we have a Daher method. We have $(\mathcal{X}_0, \mathcal{X}_1)_{z} = (X_0, X_1)_{\Real(z)}$, the complex interpolation space of Calderón \cite{Calderon1964}, with equivalence of norms. We therefore recover Daher's theorem.

\subsection{Rademacher and $\gamma$ methods} We now describe the Rademacher method of \cite{H_infty_calculus}. Fix $p \in (1, \infty)$ and let $(\varepsilon_k)_{k \in \mathbb{Z}}$ be a sequence of independent Rademacher random variables on a probability space $(\Omega, P)$. Consider the sequence structures $\mathfrak{S}_j = \epsilon^p(\mathbb{Z}, X_j)$ of all $\vec{x} \in \ell^0(\mathbb{Z}, X_j)$ given by
\[
\|\vec{x}\|_{\epsilon^p(\mathbb{Z}, X_j)} = \|\sum\limits_{k \in \mathbb{Z}} \varepsilon_k x_k\|_{L^p(\Omega, X_j)} < \infty
\]
The interpolation space $(\mathcal{X}_0, \mathcal{X}_1)_z$ is the Rademacher interpolation space $(X_0, X_1)_{z, \epsilon}$. The Rademacher method is Daher.

The $\gamma$ interpolation method is defined as the Rademacher method, but instead of Rademacher variables we take Gaussian ones. Since the Rademacher and $\gamma$ methods agree for spaces of finite cotype \cite{Euclidian_structures}, Daher's theorem is also valid for the $\gamma$ method (see the comments after Definition 10.17 and Proposition 10.39 of \cite{Pisier_Martingales}).

\subsection{$\alpha$-method} The $\alpha$-method is defined in \cite{Euclidian_structures} through the notion of Euclidean structures. Whether the $\alpha$-method is Daher or not depends on the Euclidean structure being considered. For example, the Gaussian Euclidean structure gives the $\gamma$-method, and therefore Daher's Theorem applies. However, if we take the operator norm Euclidean structure on $\ell_2$ uniform convexity is lost.

\medskip

As mentioned in the introduction, all the methods above satisfy a reiteration theorem with the complex method. More precisely, if $(X_0, X_1)$ is a compatible couple of Banach spaces let us denote by $[X_0, X_1]_{\theta}$ its complex interpolation space at $\theta \in (0, 1)$. According to \cite[Example 6.6]{DiscreteFramework}, if $(\mathcal{X}_0, \mathcal{X}_1)$ is a compatible couple of $c_0$-sequentially structured Banach spaces such that there is a constant $C > 0$ for which $\norm{(e^{iks}x_k)}_{\mathfrak{S}_j} \leq C \norm{\vec{x}}_{\mathfrak{S}_j}$ for every $\vec{x} \in \mathfrak{S}_j$, $j = 0, 1$ and $s \in \mathbb{R}$, then we have
\[
[(\mathcal{X}_0, \mathcal{X}_1)_{\theta_0}, (\mathcal{X}_0, \mathcal{X}_1)_{\theta_1}]_{\theta} = (\mathcal{X}_0, \mathcal{X}_1)_{w}
\]
with equivalence of norms for every $0 < \theta_0 < \theta_1 < 1$ and $\theta \in (0, 1)$ with $w = (1 - \theta) \theta_0 + \theta \theta_1$. As such, we could have already obtained uniform homeomorphisms between the unit spheres of interpolation spaces generated by such methods by passing through complex interpolation. The following is a possible example where reiteration is not available.

\subsection{A James' interpolation method} Given a Banach space $X$ consider a vector valued James space $J(\mathbb{Z}, X)$ which is the completion of $c_{00}(X)$ with respect to the norm
\[
\|\vec{x}\|_{J(\mathbb{Z}, X)} \coloneqq \frac{1}{\sqrt{2}} \sup \Big(\|x_{p_1} - x_{p_2}\|_X^2 + \|x_{p_2} - x_{p_3}\|_X^2 + \cdots + \|x_{p_n} - x_{p_{n-1}}\|_X^2 \Big)^{\frac{1}{2}}
\]
where the supremum is over all finite sequences of integers $p_1 < p_2 < \cdots < p_n$.

Let $J_2(\mathbb{Z}, X) \coloneqq [J(\mathbb{Z}, X), \ell_2(\mathbb{Z}, X)]_{\frac{1}{2}}$. Then $J_2(\mathbb{Z}, X)$ is a reflection invariant $c_0$-sequence structure on $X$. This may be proved used the properties of complex interpolation and the comments at the beginning of \cite[Section 2]{DiscreteFramework}. Furthermore, if $X$ is uniformly convex then so is $J_2(\mathbb{Z}, X)$ and therefore we may use it to define a Daher method. Let us show that there is no $C > 0$ such that $\norm{(e^{iks}x_k)}_{J_2(\mathbb{Z}, X)} \leq C \norm{\vec{x}}_{J_2(\mathbb{Z}, X)}$ for every $\vec{x} \in J_2(\mathbb{Z}, X)$ and $s \in \mathbb{R}$.

Fix $x \in X$ and $x^* \in X^*$ both of norm $1$ such that $x^*(x)=1$. For each $n \geq 1$ we will consider the following vectors: $\vec{x^n} = \sum\limits_{j=1}^{2n} e_j \otimes x$, $\vec{x_s^n} = \sum\limits_{j=1}^{2n} e^{ijs} e_j \otimes x$ and $\vec{z^n} = \sum\limits_{j=1}^{2n} (-1)^j e_j \otimes x^*$.

It is easy to check that $\|\vec{x^n}\|_{J(\mathbb{Z}, X)} = 1$, and therefore $\|\vec{x^n}\|_{J_2(\mathbb{Z}, X)} \leq  \sqrt{2} n^{\frac{1}{4}}$. On the other hand, $\|\vec{z^n}\|_{J(\mathbb{Z}, X)^*} \leq \sqrt{2} \sqrt{n}$, since given $\vec{y} \in J(\mathbb{Z}, X)$ of norm 1, we have:
\begin{eqnarray*}
    \abs{\vec{z^n} (\vec{y})} & = & \abs{\sum\limits_{j=1}^{2n} (-1)^j x^*(y_j)} \leq \sum\limits_{j=1}^{n} \norm{y_{2j} - y_{2j-1}}_X \leq \sqrt{2} \sqrt{n}
\end{eqnarray*}
If $X$ is reflexive then $J_2(\mathbb{Z}, X)^* = [J(\mathbb{Z}, X)^*, \ell_2(\mathbb{Z}, X)]_{\frac{1}{2}}$ isometrically. It follows that $\|\vec{z^n}\|_{J_2(\mathbb{Z}, X)^*} \leq \sqrt{2} \sqrt{n}$ and
\begin{eqnarray*}
    \|\vec{x_s^{n}}\|_{J_2(\mathbb{Z}, X)} & \geq & \frac{1}{\sqrt{2}\sqrt{n}} \abs{\vec{z^n} \vec{x_s^n}} = \frac{1}{\sqrt{2}\sqrt{n}} \left|\frac{e^{2nis}-1}{e^{is} + 1}\right|
\end{eqnarray*}
If we let $s \rightarrow \pi$ the right side approaches $\sqrt{2} \sqrt{n}$. 
Thus, there is no $C > 0$ such that $\|\vec{x_s^n}\|_{J_2(\mathbb{Z}, X)} \leq C \|\vec{x^n}\|_{J_2(\mathbb{Z}, X)}$ for every $s \in \mathbb{R}$ and every $n \geq 1$.

\section{Final remarks}

Daher's original statement of his Theorem asks only that one of the spaces $X_0$ or $X_1$ is uniformly convex. This is a simple consequence of the reiteration theorem for complex interpolation and the fact that it is enough that one of the spaces $X_0$ or $X_1$ be uniformly convex for the complex interpolation space $X_{\theta}$ to be uniformly convex. This last fact is proved through the interpolation formula $
\|x\|_{\theta} \leq \|f(i \cdot)\|_{L_p(\nu_0)}^{1 - \theta} \|f(1 + i \cdot)\|_{L_p(\nu_1)}^{\theta}
$, where $f$ is a function in the Calderón space such that $f(\theta) = x$. There is a similar formula for the discrete framework \cite[Lemma 3.8]{DiscreteFramework}, but with a multiplicative constant which prevents a direct adaptation of the proof. It is interesting to note that if $X_0$ is uniformly convex then the real interpolation space $X_{\theta, p}$ ($1 < p < \infty$) has an equivalent norm which is uniformly convex \cite[page 71]{Beauzamy}, and therefore Daher's Theorem holds for the real method asking only that $X_0$ is uniformly convex.


Daher's Theorem tells us something about how the spheres of the interpolation spaces are transformed by the interpolation method. Closely related to that is the behaviour of this transformation with respect to the Kadets metric. Recall that the Kadets distance $d_K(X, Y)$ between $X$ and $Y$ is the infimum of the Hausdorff distance between the closed unit balls $B_X$ and $B_Y$ when we consider all possible linear isometric embeddings of $X$ and $Y$ in any Banach space $Z$.

\begin{lemma}\label{lem:derivative}
Let $(\mathcal{X}_0, \mathcal{X}_1)$ be a compatible couple of sequentially structured Banach spaces and $s \in \mathbb{S}$. There is $C > 0$ such that whenever $f \in \ker \delta_s$ we have $\|g\|_{\mathcal{H}_{\pi}^2(\mathbb{S}, \mathcal{X}_0, \mathcal{X}_1)} \leq C \|f\|_{\mathcal{H}_{\pi}^2(\mathbb{S}, \mathcal{X}_0, \mathcal{X}_1)}$, where $g(z) = \frac{f(z)}{e^z - e^s}$.
\end{lemma}
\begin{proof}
Take $t \in \mathbb{S} \setminus \{s\}$ and write $f(z) = \sum\limits_{k \in \mathbb{Z}} e^{k(z-t)} y_k$. One may check that if $\Real(s) < \Real(z)$ then
\begin{eqnarray*}
g(z) & = & e^{-t} \sum\limits_{m \in \mathbb{Z}} e^{m(z-t)} \sum\limits_{n \geq 0} e^{-n(t-s)} y_{m + n + 1}
\end{eqnarray*}
and therefore
\[
\hat{g}_t(k) = e^{-t} \sum\limits_{n \geq 0} e^{-n(t-s)} y_{k+n+1}
\]
Since $f(s) = 0$, we have $\sum\limits_{k \in \mathbb{Z}} e^{k(s-t)} y_k = 0$. Therefore:
\[
e^{-kt} \hat{g}_t(k) = -\sum\limits_{n < 0} e^{ns} e^{-(n+k+1) t} y_{k+n+1}
\]
and $(e^{-kt} \hat{g}_t(k))_k = -\sum\limits_{n < 0} e^{ns} S^n ((e^{-kt} y_k)_k)$, where $S$ is the shift operator $S((x_n)_{n\in\mathbb{Z}}) = (x_{n+1})_{n\in\mathbb{Z}}$. Also,
\[
e^{k(1-t)} \hat{g}_t(k) = e^{-1} \sum\limits_{n \geq 0} e^{n(s-1)} e^{(k+n+1)(1-t)} y_{k+n+1}
\]
and therefore $(e^{k(1-t)} \hat{g}_t(k)) = e^{-1} \sum\limits_{n \geq 0} e^{n(s-1)} S^n((e^{k(1-t)} y_k)_k)$. Since $\|S\| = 1$ in any sequence structure, the result follows.
\end{proof}

Let us call the constant $C$ of the previous lemma $C_s$. It is clear that $\sup_{a < \Real(s) < b} C_s < \infty$ for every $0 < a < b < 1$.

\begin{lemma}
Let $(\mathcal{X}_0, \mathcal{X}_1)$ be a compatible couple of sequentially structured Banach spaces and $s, t \in \mathbb{S}$. For any $f \in \ker \delta_s$ there is $h \in \ker \delta_t$ such that $\|f - h\|_{\mathcal{H}_{\pi}^2(\mathbb{S}, \mathcal{X}_0, \mathcal{X}_1)} \leq \abs{e^t - e^s} C_s \|f\|_{\mathcal{H}_{\pi}^2(\mathbb{S}, \mathcal{X}_0, \mathcal{X}_1)}$. In particular, if $\|f\|_{\mathcal{H}_{\pi}^2(\mathbb{S}, \mathcal{X}_0, \mathcal{X}_1)} = 1$ there is $h \in \ker \delta_t$ of norm $1$ such that $\|f - h\|_{\mathcal{H}_{\pi}^2(\mathbb{S}, \mathcal{X}_0, \mathcal{X}_1)} \leq 2 \abs{e^t - e^s} C_s$.
\end{lemma}
\begin{proof}
Let $f \in \ker \delta_s$ and take $g \in _{\mathcal{H}_{\pi}^2(\mathbb{S}, \mathcal{X}_0, \mathcal{X}_1)}$ such that $f = (e^z - e^s) g$. Consider the function $(e^z - e^t) g \in \ker \delta_t$. Then $\|f - (e^z - e^t)g\|_{\mathcal{H}_{\pi}^2(\mathbb{S}, \mathcal{X}_0, \mathcal{X}_1)} \leq \abs{e^t - e^s} C_s \|f\|_{\mathcal{H}_{\pi}^2(\mathbb{S}, \mathcal{X}_0, \mathcal{X}_1)}$.

Suppose that $\|f\| = 1$. We have
\begin{eqnarray*}
\|f - \frac{(e^z - e^t)g}{\|(e^z - e^t)g\|}\| & \leq & \abs{e^t - e^s} C_s + \|(e^z - e^t)g - \frac{(e^z - e^t)g}{\|(e^z - e^t)g\|}\| \\
& = & \abs{e^t - e^s} C_s + \abs{\|(e^z - e^s)g\| - \|(e^z - e^t)g\|} \\
& \leq & \abs{e^t - e^s} C_s + \|(e^t - e^s) g\| \\
& \leq & 2 \abs{e^t - e^s} C_s
\end{eqnarray*}
\end{proof}

Now from \cite[Theorem 4.1]{DistancesBanach} we get:
\begin{corollary}
Let $(\mathcal{X}_0, \mathcal{X}_1)$ be a compatible couple of sequentially structured Banach spaces and $s, t \in \mathbb{S}$. We have $d_K(\mathcal{X}_{t}^{(2)}, \mathcal{X}_{s}^{(2)}) \leq 4 \abs{e^t - e^s} \max\{C_s, C_t\}$.
\end{corollary}

It follows that methods given by the discrete framework satisfy a series of stability results with respect to the Kadets distance. For example, if $X_w^{(2)}$ is either reflexive, super-reflexive or separable for some $w \in \mathbb{S}$ then $X_z^{(2)}$ is so for every $z \in \mathbb{S}$. See \cite{DistancesBanach} for more details.

Also, Lemma \ref{lem:derivative}, which is based on \cite[Lemma 3.11]{CKMR}, is the basis for a theory of commutator estimates, which is the theme of the first open question of \cite{DiscreteFramework}.

\bibliographystyle{amsplain}
\bibliography{refs}

\end{document}